\documentclass[11pt, reqno]{amsart}
\usepackage{amsrefs, amsmath, amssymb, amsthm, graphicx, marvosym}
\usepackage{cancel}
\usepackage{color}
\usepackage{bbold}

\def\veps{\varepsilon}
\def\vp{\varphi}

\def\eq#1{(\ref{#1})}
\def\nn{\nonumber}
\def\({\left(\begin{array}{cccccc}}
\def\){\end{array}\right)}

\def\eq#1{(\ref{#1})}
\def\nn{\nonumber}

\def\({\left(\begin{array}{cccccc}}
\def\){\end{array}\right)}

\def\bes{\begin{eqnarray}}
\def\ees{\end{eqnarray}}

\newcommand{\lea}{\lesssim}

\newcommand{\del}{\partial}

\newcommand{\beq}{\begin{equation}}
\newcommand{\eeq}{\end{equation}}
\newcommand{\bea}{\begin{eqnarray}}
\newcommand{\eea}{\end{eqnarray}}
\newcommand{\beann}{\begin{eqnarray*}}
\newcommand{\eeann}{\end{eqnarray*}}

\newcommand{\RR}{\mathbb{R}}

\newcommand{\EE}{\mathbb{E}}

\newcommand{\ti}{\tilde}

\newcommand{\bp}{\begin{proof}}
\newcommand{\ep}{\end{proof}}

\newtheorem{theorem}{Theorem}[section]
\newtheorem{proposition}[theorem]{Proposition}

\newtheorem{lemma}[theorem]{Lemma}

\newtheorem{remark}[theorem]{Remark}

\newtheorem{notation}{Notation}
\numberwithin{equation}{section}

\begin{document}

\title[Convergence of exterior solutions]
{On convergence of exterior solutions to radial Cauchy solutions for $\square_{1+3}U=0$}

\author[Helge Kristian Jenssen]{Helge Kristian Jenssen}\thanks{The work of 
Jenssen was supported in part by the National Science Foundation under Grant DMS-1311353}

\address[Helge Kristian Jenssen]{\newline
Department of Mathematics, Penn State University, University Park, State College, PA 16802, USA.
Email: jenssen@math.psu.edu}

\author[Charis Tsikkou]{Charis Tsikkou}\thanks{The work of Tsikkou was supported in part by 
the WVU ADVANCE Sponsorship Program}
\address[Charis Tsikkou]{\newline
Department of Mathematics, West Virginia University, Morgantown, WV 26506, USA.
Email: tsikkou@math.wvu.edu}

\date{\today}

\begin{abstract}
	Consider the Cauchy problem for the 3-d linear wave equation 
	$\square_{1+3}U=0$ with radial initial data $U(0,x)=\Phi(x)=\vp(|x|)$, 
	$U_t(0,x)=\Psi(x)=\psi(|x|)$. A standard result gives that $U$ belongs
	to $C([0,T];H^s(\RR^3))$ whenever  $(\Phi,\Psi)\in H^s\times H^{s-1}(\RR^3)$. 
	In this note we are interested in the question of how $U$ can be realized 
	as a limit of solutions to initial-boundary value problems on the 
	exterior of vanishing balls $B_\veps$ about the origin. We note that, as the solutions 
	we compare are defined on different domains, the answer is not an immediate 
	consequence of $H^s$ well-posedness for the wave equation.
	
	We show how explicit solution formulae yield convergence and optimal 
	regularity for the Cauchy solution via exterior solutions, when the latter are 
	extended continuously as constants on $B_\veps$ at each time. 
	We establish that for $s=2$ the solution $U$ can be realized 
	as an $H^2$-limit (uniformly in time) of exterior solutions on 
	$\RR^3\setminus B_\veps$ satisfying vanishing Neumann conditions
	along $|x|=\veps$, as $\veps\downarrow 0$. Similarly for $s=1$: $U$ is then an 
	$H^1$-limit of exterior solutions satisfying vanishing Dirichlet 
	conditions along $|x|=\veps$.   
\end{abstract}

\maketitle



\tableofcontents
\begin{notation}
	We use the notations $\RR^+ =(0,\infty)$ and $\RR_0^+=[0,\infty)$. 
	For function of time and spatial position, the time variable $t$ 
	is always listed first, and the spatial variable ($x$ or $r$) is listed last. 
	We indicate by subscript ``$rad$'' that the functions under consideration 
	are spherically symmetric, e.g.\ $H^2_{rad}(\RR^3)$ denotes the set 
	of $H^2(\RR^3)$-functions $\Phi$ with the property that 
	$\Phi(x)=\vp(|x|)$ for some function $\vp:\RR_0^+\to\RR$. For a radial 
	function we use the same symbol whether it is considered as a function 
	of $x$ or of $r=|x|$.
	
	Throughout we fix $T>0$ and $c>0$ and  set 
	\[\square_{1+1}:=\del_t^2-c^2\del_r^2,
	\qquad \square_{1+3}:=\del_t^2-c^2\Delta,\]
	where $\Delta$ is the 3-d Laplacian. 
	The open ball of radius $r$ about the origin in $\RR^3$ is denoted $B_r$.
	We write $\del_r$ for the directional derivative in the (outward) radial direction
        while $\del_i$ denotes $\del_{x_i}$. 	
	Finally, for two functions $A(t)$ and $B(t)$ we write 
        \[A(t)\lea B(t)\]
        to mean that there is a number $C$, possibly depending 
        on the time $T$, $c$, the fixed cutoff functions $\beta$ and $\chi$ (see 
        \eq{approx_neum_data_1}-\eq{approx_neum_data_2} and
        \eq{approx_dir_data_1}-\eq{approx_dir_data_2}), as well as 
        the initial data $\Phi$, $\Psi$, but independent of the vanishing radii 
        $\veps$, such that
        \[A(t)\leq C\cdot B(t)\qquad\text{holds for all $t\in [0,T]$.}\]
\end{notation}
\section{Radial Cauchy solutions as limits of exterior solutions}\label{rad_solns}
Consider the Cauchy problem for the 3-d linear wave equation 
with radial initial data:
\[\text{(CP)}\qquad
\left\{\begin{array}{ll}
	\square_{1+3}U=0 & \text{on $(0,T)\times\RR^3$}\\
	U(0,x)=\Phi(x) & \text{on $\RR^3$}\\
	U_t(0,x)=\Psi(x) & \text{on $\RR^3$,}
\end{array} \right.\]
where 
\[\Phi\in H^s_{rad}(\RR^3), \qquad \Psi\in H^{s-1}_{rad}(\RR^3),\]
with 
\beq\label{rad_versns}
	\Phi(x)=\vp(|x|)\qquad \Psi(x)=\psi(|x|).
\eeq
Throughout we refer to the unique solution $U$ of (CP) as the {\em Cauchy 
solution}. 

In this work  we consider how the radial Cauchy solution $U$ can be realized 
as a limit of solutions to initial-boundary value problems posed on the
exterior of vanishing balls $B_\veps$ ($\veps\downarrow 0$) about the origin. 
The precise issue will be formulated below.
We shall consider exterior solutions satisfying either a vanishing Neumann 
condition or a vanishing Dirichlet condition along $|x|=\veps$.

It is well known that the Sobolev spaces $H^s$ provide a natural 
setting for the Cauchy problem for the wave equation; see \cite{ra} and 
\eq{opt_reg_1}-\eq{opt_reg_2} below. 
The choice of space dimension $3$ is for convenience: it is particularly
easy to generate radial solutions in this case. Next, both the choice of 
spaces for the initial data for (CP), as well as the boundary condition 
imposed on the exterior solutions, will influence the convergence of
exterior solutions toward the Cauchy solution. 
For the wave equation in $\RR^3$ the different convergence behavior of
exterior Neumann and exterior Dirichlet solutions is brought out by 
considering $H^2$ vs.\ $H^1$ initial data; see Remark \ref{1st_rmk} below.

The scheme of generating radial solutions to Cauchy problems as limits of 
exterior solutions has been applied to a variety of evolutionary PDE problems; 
see \cite{jt1} for references and discussion. 
In our earlier work \cite{jt1} we used the 3-d wave equation to gauge
the effectiveness of this general scheme in a case where 
``everything is known.'' In order that the results be relevant
to other (possibly nonlinear) problems, the analysis in \cite{jt1} deliberately 
avoided any use of explicit solution formulae. Based on energy arguments 
and strong convergence alone, it was found that the exterior solutions do
converge to the Cauchy solution as the balls vanish. However,
the arguments did not yield optimal information about the regularity of the
limiting Cauchy solution. Specifically,
for $s=2$ we obtained the Cauchy solution as a limit only in 
$H^1$ (via exterior Neumann solutions) or in $L^2$ (via exterior Dirichlet
solutions). This is strictly less regularity than what is known to be the case, see 
\eq{opt_reg_2}.  

Thus, in general, while limits of exterior solutions to evolutionary PDEs 
may be used to establish existence for radial Cauchy problems, one should 
not expect optimal regularity information about the Cauchy solution via this 
approach. 

On the other hand, for the particular case of the 3-d wave equation 
with radial data, it is natural to ask what type of convergence we can
establish if we exploit solution formulae (for the Cauchy 
solution as well as for the exterior solutions). The present work 
addresses this question, and our findings are summarized 
in Theorem \ref{main_result} below.

We stress that while \cite{jt1} dealt with the issue of using exterior 
solutions as a stand-alone method for obtaining existence of radial
Cauchy problems, the setting for the present work is different. 
We are now exploiting what is known about the solution of the
Cauchy solutions as well as exterior solutions for the radial 3-d 
linear wave equation, and the only issue is how the former 
solutions are approximated by the latter.

\begin{remark}\label{1st_rmk}
	Before starting the detailed analysis we comment on a
	slightly subtle point. As recorded in our main result (Theorem 
	\ref{main_result}), we 
	establish $H^2$-regularity of the limiting Cauchy solution $U$
	when the initial data $(\Phi,\Psi)$ belong to $H^2\times H^1$, 
	and $H^1$-regularity when the data belong to $H^1\times L^2$. 
	This is as it should be according to 
	\eq{opt_reg_1}. Now, in the former case $U$
	is obtained as a limit of exterior Neumann solutions, while in the 
	latter case it is obtained as a limit of exterior Dirichlet solutions.
	This raises a natural question: what regularity is obtained for 
	$U$ in the case of $H^2\times H^1$-data, if we insist on 
	approximating by exterior Dirichlet solutions?
	
	To answer this we need to specify how we compare the 
	everywhere defined Cauchy solution $U$ to exterior solutions 
	$U^\veps$, which are defined only on the exterior domains 
	$\RR^3_\veps:=\RR^3\setminus B_\veps$. 
	There are at least two ways to do this\footnote{When using 
	exterior solutions to establish existence for (CP) (as in \cite{jt1}), 
	there is no such choice: one must produce approximations to 
	$U$ that are everywhere defined.}:
	\begin{itemize}
		\item[(a)] by calculating $\|U(t)-U^\veps(t)\|_{H^s(\RR^3_\veps)}$;
		\item[(b)] by first defining a suitable extension $\tilde U^\veps$ of $U^\veps$
		to all of $\RR^3$, and then calculating 
		$\|U(t)-\tilde U^\veps(t)\|_{H^s(\RR^3)}$.
	\end{itemize}
	With {\em (b)}, which is what we do in this paper, the natural choice 
	is to extend $U^\veps(t)$ continuously as a constant on $B_\veps$ 
	at each time. I.e., for  
	exterior Dirichlet solutions, we let $\ti U^\veps(t,x)$ vanish identically on 
	$B_\veps$, while for exterior Neumann solutions its value there 
	is that of $U^\veps(t,x)$ along the $|x|=\veps$.
	
	It turns out that regardless of whether we use  {\em (a)} or {\em (b)} to compare 
	the Cauchy solution to the exterior solutions, the answer to the question 
	above is that we obtain only $H^1$-convergence when 
	exterior Dirichlet solutions are used. In fact, for {\em (b)} this is immediate: 
	the exterior Dirichlet solution $\tilde U^\veps$ will typically have a nonzero 
	radial derivative at $r=\veps+$ so that its extension $\tilde U^\veps$ 
	contains a ``kink'' along $|x|=\veps$. Thus, second derivatives of 
	$\tilde U^\veps$ will typically contain a $\delta$-function along $|x|=\veps$, 
	and $\tilde U^\veps$ does not even belong to $H^2(\RR^3)$ in this case.
	For {\em (a)} it suffices to consider the situation at time zero. 
	With $\Phi$ as above we consider smooth cutoffs $\Phi^\veps$ 
	(see \eq{approx_dir_data_1} below). 
	A careful calculation, carried out in \cite{jt1}, shows that 
	$\|\Phi-\Phi^\veps\|_{H^2(\RR^3_\veps)}$ blows up as 
	$\veps\downarrow 0$.
	
	These remarks highlight the unsurprising but relevant fact that 
	exterior Dirichlet  solutions are more singular than exterior Neumann 
	solutions; see \cite{jt1} for a discussion.
\end{remark}

The goal is to show that the Cauchy solution $U$ of (CP) can be approximated,
uniformly on compact time intervals, in $H^2$-norm by suitably chosen 
exterior Neumann solutions and in $H^1$-norm by 
suitably chosen exterior Dirichlet solutions.

As indicated we shall use explicit solution formulae for both the Cauchy 
problem (CP) as well as for the exterior Neumann and Dirichlet problems.
These formulae for radial solutions are readily available in 3 dimensions
and exploits the fact that radial solutions of $\square_{1+3}U=0$ 
admit the representation 
\[U(t,x)=\frac{u(t,|x|)}{|x|}\]
where $u(t,r)$ solves $\square_{1+1}u=0$ on the half-line $\RR^+$.
(Exterior Neumann solutions require a little work to write down
explicitly; see \eq{U_eps}.) 

Of course, with the explicit formulae in place, it is a matter of 
computation to analyze the required norm differences. However, 
it is a rather involved computation since the formulae involve different 
expressions in several different regions. 
Also, the answers do not follow by appealing to well-posedness
for the wave equation (see \eq{opt_reg_2} below): the Cauchy solution 
and the exterior solutions are defined on different domains. 
As noted above we opt to extend the exterior solutions to the interior 
of the balls $B_\veps$, before comparing them to the Cauchy solution. 

Instead of a direct comparison we prefer to estimate the 
$H^2$- and $H^1$-differences in question by employing the 
natural energies for the wave equation. These energies will majorize 
the $L^2$-distances of the first and second derivatives, and will 
also provide an estimate on the $L^2$-distance of the functions
themselves.

There are two advantages of this approach: first, it is straightforward 
to calculate the exact rates of change of the energies in question, and second, 
these rates depend only on what takes place at or within radius $r=\veps$.
The upshot is that it suffices to analyze fewer terms than required by
a direct approach. Finally, to estimate the rates of change of the relevant 
energies we make use of the explicit solution formulae.

\section{Setup and statement of main result}
\subsection{The Cauchy solution}
A standard result (see e.g.\ \cites{bjs,ra}) shows that the radial Cauchy solution $U$ 
of (CP) may be calculated explicitly by using the representation 
\[U(t,x)=\frac{u(t,|x|)}{|x|},\]
where $u(t,r)$ solves the half-line problem
\beq\label{half_line_dir}
	\text{(Half-line)}\qquad
	\left\{\begin{array}{ll}
	\square_{1+1}u =0 & \text{on $(0,T)\times\RR^+$}\\
	u(0,r)=r\vp(r) & \text{for $r\in\RR^+$}\\
	u_t(0,r)=r\psi(r) & \text{for $r\in\RR^+$}\\
	u(t,0)\equiv 0 & \text{for $t>0$,}\end{array} \right.
\eeq
where $\vp$ and $\psi$ are as in \eq{rad_versns}.
By using the d'Alembert formula for the half-line problem (see \cite{bjs}) 
we obtain that
\beq\label{U}
	U(t,r)=\left\{
	\begin{array}{ll}
		0\leq r\leq ct: & \frac{1}{2r}\left[ (ct+r) \vp(ct+r)-(ct-r)\vp(ct-r)\right]\\
					& +\frac{1}{2cr}\int_{ct-r}^{ct+r} s\psi(s)\, ds\\\\
		r\geq ct: & \frac{1}{2r}\left[ (r+ct) \vp(r+ct)+(r-ct)\vp(r-ct)\right]\\
					& +\frac{1}{2cr}\int_{r-ct}^{r+ct} s\psi(s)\, ds.
	\end{array}
	\right.
\eeq
In addition to the solution formula \eq{U} we shall also exploit the 
following well-known stability property \cites{ra,sel}: 
with data $\Phi\in H^s(\RR^3)$ and $\Psi\in H^{s-1}(\RR^3)$ 
(radial or not), the Cauchy problem (CP) admits a unique solution 
$U$ which satisfies
\beq\label{opt_reg_1}
	U\in C([0,T];H^s(\RR^3))\cap C^1([0,T];H^{s-1}(\RR^3))
\eeq
and 
\beq\label{opt_reg_2}
	\|U(t)\|_{H^s(\RR^3)}+\|U_t(t)\|_{H^{s-1}(\RR^3)}
	\leq C_T\left( \|\Phi\|_{H^s(\RR^3)}+\|\Psi\|_{H^{s-1}(\RR^3)}\right),
\eeq
for each $T>0$, where $C_T$ is a number of the form 
$C_T=\bar C\cdot(1+T)$, and $\bar C$ a universal constant.

\subsection{Exterior solutions and their extensions}\label{ext_solns}
With $\Phi\in H^s_{rad}(\RR^3)$ and $\Psi\in H^{s-1}_{rad}(\RR^3)$, $s=1$ 
or $2$, the goal is to show that the  solution $U$ of (CP) can be ``realized as a limit of 
exterior solutions'' defined outside of $B^\veps$ as $\veps\downarrow 0$. 
To make this precise we need to specify:
\begin{enumerate}
	\item precisely which exterior solutions we consider: which boundary 
	conditions do they satisfy along $\del B_\veps$, and how are their 
	initial data related to the given Cauchy data $\Phi$, $\Psi$;
	\item how we compare the everywhere defined Cauchy solution $U$ 
	with exterior solutions $U^\veps$, that are defined only outside of 
	$B_\veps$; and
	\item which norm we use for comparing $U$ and $U^\veps$.
\end{enumerate}
Concerning (1) we shall consider exterior solutions that satisfy either 
vanishing Neumann or vanishing Dirichlet conditions along $\del B^\veps$. 
In either case, the initial data for the exterior problem are generated 
by a two-step procedure: we first approximate the original 
Cauchy data by $C^\infty_{c,rad}(\RR^3)$-functions, 
and then apply an appropriate modification of these 
smooth approximations near the origin. These  modifications
use smooth cut-off functions and are made so that the result 
satisfies vanishing Neumann or Dirichlet conditions along $|x|=\veps$. 
(See \eq{approx_neum_data_1}-\eq{approx_neum_data_2} and
\eq{approx_dir_data_1}-\eq{approx_dir_data_2} for details.)
In either case we denote the exterior, radial solutions corresponding to 
the approximate, smooth data by $U^\veps(t,x)\equiv U^\veps(t,r)$; 
they are given explicitly in \eq{U_eps} and \eq{U_eps_dir} below.

As mentioned in Remark \ref{1st_rmk}, for (2) we opt to compare 
the Cauchy solution $U$ to the natural extensions $\ti U^\veps$ of 
the smooth exterior solution $U^\veps$: at each time $t$, 
$\ti U^\veps(t,x)$ takes the constant value $U^\veps(t,\veps)$ on 
$B_\veps$, and coincides with $U^\veps(t,x)$ for $|x|\geq \veps$. 
Thus, in the case of Dirichlet data, $\ti U^\veps(t,x)$ vanishes 
identically on $B_\veps$, while for Neumann data its value there 
is that of $U^\veps(t,x)$ along the boundary $|x|=\veps$. 

Finally, concerning (3), Remark \ref{1st_rmk} above also explains  
the choice of $H^2$-norm for comparing the Cauchy solution $U$ 
to exterior Neumann solutions, and $H^1$-norm for comparison 
to exterior Dirichlet solutions. 

Our main result is as follows.

\begin{theorem}\label{main_result}
	Let $T>0$ be given and let $U$ denote the solution of the 
	radial Cauchy problem $\mathrm{(CP)}$ for the linear wave 
	equation in three space dimensions with initial data $(\Phi,\Psi)$.
	\begin{itemize}
		\item[(i)] For initial data in $H^2_{rad}(\RR^3)\times H^1_{rad}(\RR^3)$ the 
		Cauchy solution $U$ can be realized as a $C([0,T];H^2(\RR^3))$-limit 
		of suitable extended exterior Neumann solutions as 
		$\veps\downarrow 0$.
		\item[(ii)] For initial data in $H^1(\RR^3)\times L^2(\RR^3)$ the 
		Cauchy solution $U$ can be realized as a $C([0,T];H^1(\RR^3))$-limit 
		of suitable extended exterior Dirichlet solutions as 
		$\veps\downarrow 0$.
	\end{itemize}
\end{theorem}
We point out that, e.g.\ in part (i), we do not claim that the 
extended Neumann solutions $\tilde U^\veps$ converge to $U$ 
in $H^2$-norm. In fact, we establish this latter property only for 
the case with $C^\infty_c(\RR^3)$ initial data. However, thanks 
to the stability property \eq{opt_reg_2}, this is sufficient to obtain 
(i); see Proposition \ref{smooth_case} below.

The rest of the paper is organized as follows. After reducing to 
the case with smooth and compactly supported data in Section 
\ref{smooth_c}, we treat $H^2$-convergence of exterior Neumann 
solutions in Sections \ref{gen_neu_soln}-\ref{comp_neu}, while
$H^1$-convergence of exterior Dirichlet solutions is established 
in Sections \ref{gen_dir_soln}-\ref{comp_dir}.

\subsection{Reduction to smooth case}\label{smooth_c}
The first step of the proof is to use well-posedness \eq{opt_reg_2} 
for the Cauchy problem to reduce to the case of smooth initial data. 
\begin{proposition}\label{smooth_case}
	With the setup in Theorem \ref{main_result}, 
	let $\tilde U^{N,\veps}$ and $\tilde U^{D,\veps}$ denote the 
	extensions of the exterior Neumann and Dirichlet solutions, 
	respectively, as described in Section \ref{ext_solns}.
	Then, Theorem \ref{main_result} follows once it is established that 
	\beq\label{smooth_case1}
		\sup_{0\leq t\leq T}\|U(t)- \tilde U^{N,\veps}(t)\|_{H^2(\RR^3)}\to 0
		\qquad\text{as $\veps\downarrow 0$}
	\eeq
	and  
	\beq\label{smooth_case2}
		\sup_{0\leq t\leq T}\|U(t)- \tilde U^{D,\veps}(t)\|_{H^1(\RR^3)}\to 0
		\qquad\text{as $\veps\downarrow 0$,}
	\eeq
	for any initial data $\Phi,\Psi\in C_{c,rad}^\infty(\RR^3)$.
\end{proposition}
\begin{proof}
	For concreteness consider the case of exterior 
	Neumann solutions, and let arbitrary  data $\Phi\in H^2_{rad}(\RR^3)$, 
	$\Psi\in H^1_{rad}(\RR^3)$ be given. Fix any $\delta>0$. 
	We first choose $\Phi_0$, $\Psi_0$ in $C_{c,rad}^\infty(\RR^3)$ with
	\[\|\Phi-\Phi_0\|_{H^2}+ \|\Psi-\Psi_0\|_{H^1}< \frac{\delta}{2C_T},\]
	where $C_T$ is as in \eq{opt_reg_2}.
	The existence of such $\Phi_0$, $\Psi_0$ may be established in a 
	standard manner via convolution (using a radial mollifier) 
	and smooth cutoff at large radii. Let $U_0$ denote the solution of 
	(CP) with data $\Phi_0$, $\Psi_0$. Also, for any $\veps>0$ 
	let $\tilde U^{N,\veps}_0(t,x)$ denote the extension of the exterior Neumann 
	solution with data $\Phi_0^\veps$, $\Psi_0^\veps$, as described in 
	Section \ref{ext_solns}. 
	Then, assuming that \eq{smooth_case1} has been 
	established, we can choose $\veps>0$ 
	sufficiently small to guarantee that 
	\[\sup_{0\leq t\leq T}\|U_0(t)- \tilde U^{N,\veps}_0(t)\|_{H^2}<\frac{\delta}{2}.\]
	Hence, for any $t\in[0,T]$ we have
	\begin{align*}
		\|U(t)- \tilde U^{N,\veps}_0(t)\|_{H^2}
		&\leq \|U(t)- U_0(t)\|_{H^2}+\|U_0(t)- \tilde U^{N,\veps}_0(t)\|_{H^2}\\
		&\overset{\eq{opt_reg_2}}{\leq} 
		C_T\left(\|\Phi-\Phi_0\|_{H^2}+ \|\Psi-\Psi_0\|_{H^1}\right)+\frac{\delta}{2}
		<\delta,
	\end{align*}
	by the choice of $\Phi_0$, $\Psi_0$. 
\end{proof}
From now on we therefore consider an arbitrary but fixed pair of 
functions $\Phi,\, \Psi\in C^\infty_{c,\, rad}(\RR^3)$.
Note that we then have that the functions $\vp$ and $\psi$
in \eq{rad_versns} are smooth on $\RR_0^+$ and satisfy $\vp'(0+)=\psi'(0+)=0$.

\section{Exterior Neumann solutions}\label{gen_neu_soln}
In this section and the next we consider the case of exterior Neumann 
solutions.
For the fixed initial data $\Phi,\, \Psi\in C^\infty_{c,\, rad}(\RR^3)$ and 
any $\veps>0$ we derive a formula for $U^\veps(t,x)\equiv U^{N,\veps}(t,x)$, 
defined for 
$|x|\geq\veps$ and satisfying $\del_rU^\veps|_{r=\veps}=0$.  We refer 
to $U^\veps$ as the {\em exterior Neumann solution} corresponding 
to the solution $U$ of (CP) with data $\Phi$, $\Psi$. In Section 
\ref{comp_neu} we will then estimate how it (really, its extension
$\tilde U^\veps(t,x)$ to all of $\RR^3$) approximates the solution 
$U(t)$ in $H^2(\RR^3)$ at fixed times.

To generate the exterior Neumann solution $U^\veps$ we 
fix a smooth, nondecreasing function $\beta:\RR_0^+\to\RR_0^+$ with
\beq\label{beta_props_1}
	\beta\equiv 1 \quad\text{on $[0,1]$,}\quad \beta(s)=s\quad\text{for $s\geq 2$.}
\eeq
Then, with $\vp$ and $\psi$ as in \eq{rad_versns}, we define
\beq\label{approx_neum_data_1}
	\Phi^\veps(x)\equiv \vp^\veps(|x|):=\vp\big(\veps\beta
	\big(\textstyle\frac{|x|}{\veps}\big)\big)
\eeq
and
\beq\label{approx_neum_data_2}
	\Psi^\veps(x)\equiv \psi^\veps(|x|):=\psi\big(\veps\beta
	\big(\textstyle\frac{|x|}{\veps}\big)\big).
\eeq
We refer to $(\Phi^\veps,\Psi^\veps)$ as the {\em Neumann data} 
corresponding to the original Cauchy data $(\Phi,\Psi)$ for (CP).
Note that the Neumann data are actually defined on all of $\RR^3$, 
that they are constant (equal to $\vp(\veps)$ and $\psi(\veps)$, 
respectively) on $B_\veps$, and that their restrictions to the 
exterior domain $\{x\in\RR^3\,:\, |x|\geq\veps\}$  
satisfy homogeneous Neumann conditions along $|x|=\veps$.

The exterior Neumann solution $U^\veps$ is now defined as 
the unique radial solution of the initial-boundary value problem
\[\left\{\begin{array}{ll}
	\square_{1+3}V =0 & \text{on $(0,T)\times\{|x|>\veps\}$}\\
	V(0,x)=\Phi^\veps(x) & \text{for $|x|>\veps$}\\
	V_t(0,x)=\Psi^\veps(x) & \text{for $|x|>\veps$}\\
	\del_r V(t,x)=0 & \text{along $|x|=\veps$ for $t>0$.}
\end{array} \right.\]

To obtain a formula for $U^\veps$ we exploit the fact that $V$ is a radial 
solution of the 3-d wave equation if and only if $v=rV$ solves the 1-d wave equation.
Setting 
\[u^\veps(t,r):=rU^\veps(t,r),\] 
we obtain that $u^\veps$ solves the corresponding 1-d problem on $\{r>\veps\}$:
\[\text{($\veps$-Half-line)}\qquad
\left\{\begin{array}{ll}
	\square_{1+1}u=0 & \text{on $(0,T)\times\{r>\veps\}$}\\
	u(0,r)=r\vp^\veps(r) & \text{for $r>\veps$}\\
	u_t(0,r)=r\psi^\veps(r) & \text{for $r>\veps$}\\
	u_r(t,\veps)= \frac{1}{\veps}u(t,\veps) & \text{for $t>0$.}
\end{array} \right.\]
Note that the Neumann condition for the 3-d solution corresponds to a Robin 
condition for the 1-d solution. (A direct calculation shows that the initial data 
for $u^\veps$ and $u_t^\veps$ both satisfy this Robin condition.)

The solution $u^\veps$ to the $\veps$-Half-line problem is explicitly 
given via d'Alembert's formula\footnote{One way to solve the 1-dimensional 
Robin IBVP is to first solve the IBVP with general Dirichlet data 
$u^\veps(t,\veps)=h(t)$ along $r=\veps$, for which a d'Alembert formula is 
readily available (see John \cite{bjs}, p.\ 8); one may then identify the $h$ 
which gives $u_r= \frac{1}{\veps}u$ along $r=\veps$.}:
\beq\label{u_eps}
	u^\veps(t,r)=\left\{
	\begin{array}{ll}
		\veps\leq r\leq ct+\veps: & \frac{1}{2}\left[ (ct+r) \vp^\veps(ct+r)
				+(ct-r+2\veps)\vp^\veps(ct-r+2\veps)\right]\\
				& +\frac{1}{2c}\int_{ct-r+2\veps}^{ct+r} s\psi^\veps(s)\, ds \\
				&+e^{\frac{r-ct-2\veps}{\veps}}
				\int_\veps^{ct-r+2\veps}\left[\frac{s\psi^\veps(s)}{c}
				-\frac{s\vp^\veps(s)}{\veps}\right]e^{\frac{s}{\veps}}\, ds\\\\
				r\geq ct+\veps: & \frac{1}{2}\left[ (r+ct) \vp^\veps(r+ct)
				+(r-ct)\vp^\veps(r-ct)\right]\\
				& +\frac{1}{2c}\int_{r-ct}^{r+ct} s\psi^\veps(s)\, ds.
	\end{array}\right.
\eeq
A direct calculation shows that $u^\veps$ is a classical solution on 
$\RR_t\times \{r>\veps\}$. From this we obtain the radial exterior 
Neumann solution $U^\veps(t,r):=\frac{u^\veps(t,r)}{r}$:
\beq\label{U_eps}
	U^\veps(t,r)=\left\{
	\begin{array}{ll}
		\veps\leq r\leq ct+\veps: & \frac{1}{2r}\left[ (ct+r) \vp^\veps(ct+r)
				+(ct-r+2\veps)\vp^\veps(ct-r+2\veps)\right]\\
				& +\frac{1}{2cr}\int_{ct-r+2\veps}^{ct+r} s\psi^\veps(s)\, ds\\
				&+\frac{1}{r}e^{\frac{r-ct-2\veps}{\veps}}
				\int_\veps^{ct-r+2\veps}\left[\frac{s\psi^\veps(s)}{c}
				-\frac{s\vp^\veps(s)}{\veps}\right]
				e^{\frac{s}{\veps}}\, ds\\\\
		r\geq ct+\veps: & \frac{1}{2r}\left[ (r+ct) \vp^\veps(r+ct)
				+(r-ct)\vp^\veps(r-ct)\right]\\
				& +\frac{1}{2cr}\int_{r-ct}^{r+ct} s\psi^\veps(s)\, ds.
	\end{array}
	\right.
\eeq
We finally extend $U^\veps$ at each time to obtain an everywhere 
defined approximation of the Cauchy solution $U(t,x)$. As discussed
earlier we use the natural choice of extending $U^\veps$ continuously 
as a constant on $B_\veps$ at each time:
\beq\label{tilde_U_eps}
	\tilde U^\veps(t,x)=\left\{
	\begin{array}{ll}
		U^\veps(t,\veps) & \text{for $0\leq |x|\leq \veps$}\\\\
		U^\veps(t,x) & \text{for $|x|\geq \veps$.}
	\end{array}
	\right.
\eeq
For later use we record that the value along the boundary is explicitly given as
\beq\label{U_eps_bndry_val}
	U^\veps(t,\veps)=\frac{1}{\veps}(ct+\veps)\vp^\veps(ct+\veps)
	+\frac{1}{\veps}e^{-\frac{ct+\veps}{\veps}}
	\int_\veps^{ct+\veps}\left[\frac{s\psi^\veps(s)}{c}-\frac{s\vp^\veps(s)}{\veps}\right]
	e^{\frac{s}{\veps}}\, ds,
\eeq
and we also note that 
\beq\label{init_approx_data}
	\tilde U^\veps(0,x)=\Phi^\veps(x),\qquad 
	\tilde U_t^\veps(0,x)=\Psi^\veps(x)\qquad\text{for all $x\in\RR^3$.}
\eeq

\section{Comparing Cauchy and exterior Neumann solutions}\label{comp_neu}
The issue now is to estimate the $H^2$-distance
\[\|U(t)-\tilde U^\veps(t)\|_{H^2(\RR^3)}\]
as $\veps\downarrow 0$. As explained in Section \ref{rad_solns}
we prefer to estimate this $H^2$-difference by employing the 
natural energies for the wave equation. 
These energies will majorize the $L^2$-distances of the 
first and second derivatives of $U(t)$ and $\tilde U^\veps(t)$, and  
also provide control of the $L^2$-distance of the functions themselves.

\subsection{Energies}\label{energies}
For any function $W(t,x)$ which is twice weakly differentiable on 
$\RR\times\RR^3$ we define the following 1st and 2nd order energies
(note their domains of integration):
\[\mathcal E_{W}(t):={\textstyle\frac{1}{2}}\int_{\RR^3} 
|\del_t  W(t,x)|^2+c^2|\nabla W(t,x)|^2\, dx,\]
\[\mathcal E^\veps_{W}(t):={\textstyle\frac{1}{2}}\int_{|x|\geq \veps} 
|\del_t  W(t,x)|^2+c^2|\nabla W(t,x)|^2\, dx,\]
and
\[\EE_{W}(t):=\sum_{i=1}^3 \mathcal E_{\del_i W}(t)
=\sum_{i=1}^3{\textstyle\frac{1}{2}}\int_{\RR^3} 
|\del_t \del_i W(t,x)|^2 +c^2|\nabla \del_i W(t,x)|^2\, dx,\]
\[\EE^\veps_{W}(t):=\sum_{i=1}^3 \mathcal E^\veps_{\del_i W}(t)
=\sum_{i=1}^3{\textstyle\frac{1}{2}}\int_{|x|>\veps} 
|\del_t \del_i W(t,x)|^2 +c^2|\nabla \del_i W(t,x)|^2\, dx.\]
The first goal is to estimate the energies
\beq\label{ult_energ_1}
	\mathcal E^\veps(t):=\mathcal E_{U-\ti U^\veps}(t)
\eeq
\beq\label{ult_energ_2}
	\EE^\veps(t):=\EE_{U-\ti U^\veps}(t),
\eeq
which majorizes the $L^2$-distances between the 1st and 2nd 
derivatives of $U$ and $\ti U^\veps$, respectively.
As a first step we observe the following facts.
\begin{lemma}\label{energy_1}
	With $U$ and $U^\veps$ as defined above we have: each of the energies
	\[\mathcal E_{U}(t),\quad \mathcal E^\veps_{U^\veps}(t),\quad
	\mathcal E_{\del_i U}(t),\quad\text{and}\quad\mathcal E^\veps_{\del_i U^\veps}(t)\]
	are constant in time.
\end{lemma}
\begin{proof}
	The constancy of the first three energies is standard, while the 
	constancy of $\mathcal E^\veps_{\del_i U^\veps}(t)$ is a consequence 
	of the fact that we consider radial solutions. Indeed, as $U^\veps$
	is radial and satisfies vanishing Neumann conditions along $|x|=\veps$,
	we have that $\nabla U^\veps(t,x)\equiv 0$ along $|x|=\veps$. Thus,
	$U^\veps_{x_it}\equiv 0$ for each $i=1,2,3$ along $|x|=\veps$.
	Differentiating in time, using that $U^\veps$ is a solution of the wave equation,
	and integrating by parts, we therefore have
	\begin{align*}
		\dot{\mathcal E}^\veps_{\del_i U^\veps}(t) 
		&= \int_{|x|>\veps} U^\veps_{x_it}U^\veps_{x_itt}
		+c^2\nabla U^\veps_{x_i}\cdot\nabla U^\veps_{x_it}\, dx\\
		&=c^2\int_{|x|>\veps} U^\veps_{x_it}\Delta U^\veps_{x_i}
		+\nabla U^\veps_{x_i}\cdot\nabla U^\veps_{x_it}\, dx
		= c^2\int_{\del\{|x|>\veps\}}U^\veps_{x_it}
		\frac{\del U^\veps_{x_i}}{\del\nu}\, dS=0.
	\end{align*}
\end{proof}
Next, to estimate $\mathcal E^\veps(t)$, we expand the integrand and use that 
$\nabla \ti U^\veps$ vanishes on $B_\veps$ (by our choice of extension),
to get
\begin{align*}	
	\mathcal E^\veps(t) &= \mathcal E_{U- \ti U^\veps}(t) 
	=\frac{1}{2} \int_{\RR^3} |U_t-\ti U^\veps_t|^2 + c^2|\nabla U-\nabla \ti U^\veps|^2\, dx\\
	&=\mathcal E_U(t)+\mathcal E_{\ti U^\veps}(t) 
	-\int_{\RR^3} U_t\ti U^\veps_t +c^2\nabla U\cdot\nabla \ti U^\veps\, dx\\
	&= \mathcal E_U(t)+\mathcal E^\veps_{U^\veps}(t) 
	+\frac{\text{vol}(B_\veps)}{2} |U^\veps_t(t,\veps)|^2
	-U^\veps_t(t,\veps)\int_{|x|<\veps} U_t(t,x)\, dx\\
	&\quad-\int_{|x|>\veps} U_tU^\veps_t+c^2\nabla U\cdot\nabla U^\veps\, dx.
\end{align*}
Differentiating in time, applying Lemma \ref{energy_1}, integrating by parts,
and using the boundary condition $\del_r U^\veps(t,\veps)\equiv 0$,  
then yield
\[\dot{\mathcal E}^\veps(t) = \frac{d}{dt}\left[\frac{\text{vol}(B_\veps)}{2} |U^\veps_t(t,\veps)|^2
-U^\veps_t(t,\veps)\int_{|x|<\veps} U_t(t,x)\, dx\right]
+c^2\int_{|x|=\veps} U^\veps_t \del_rU\, dS.\]
Integrating back up in time, and recalling that $U^\veps$ and $U$ are radial, we obtain
\begin{align}\label{1_st_energy_diff}
	\mathcal E^\veps(T) &= \mathcal E^\veps(0) 
	+\left[\frac{\text{vol}(B_\veps)}{2} |U^\veps_t(t,\veps)|^2 
	-U^\veps_t(t,\veps)\int_{|x|<\veps} U_t(t,x)\, dx\right]_{t=0}^{t=T}\nn\\
	&\quad + c^2\text{area}(B_\veps)\int_0^T U^\veps_t(t,\veps) \del_rU(t,\veps)\, dt.
\end{align}
Below we shall carefully estimate the terms on the RHS to show that 
$\mathcal E^\veps(T)\to 0$ as $\veps\downarrow 0$.

Before carrying out a similar representation of the 2nd order energy 
difference $\EE^\veps(t)$, we observe how $\mathcal E^\veps(t)$ 
controls the $L^2$-distance between $U$ and $\ti U^\veps$.
Setting
\beq\label{L2_dist}
	\mathcal D^\veps(t):=\frac{1}{2}\int_{\RR^3}|U(t,x)-\ti U^\veps(t,x)|^2\, dx,
\eeq
the Cauchy-Schwarz inequality gives 
\[\dot{\mathcal D}^\veps(t)\leq 2\mathcal D^\veps(t)^\frac{1}{2}
\mathcal E^\veps(t)^\frac{1}{2},\]
such that
\beq\label{1_vs_0_energy}
	\mathcal D^\veps(T)\lea \mathcal D^\veps(0)+\int_0^T \mathcal E^\veps(t)\, dt.
\eeq

We now consider how $\EE^\veps(t)$ changes in time. Arguing as above, using 
Lemma \ref{energy_1} and the fact that $U^\veps_{x_i}\equiv 0$ 
on $B_\veps$, we have
\begin{align}
	\mathcal E_{\del_i U-\del_i \ti U^\veps}(t) &=
	\frac{1}{2}\int_{\RR^3} |U_{x_it}-\ti U^\veps_{x_it}|^2
	+c^2|\nabla U_{x_it}-\nabla \ti U^\veps_{x_it}|^2\, dx\nn\\
	&=\mathcal E_{\del_i U}(t)+\mathcal E_{\del_i \ti U^\veps}(t)
	-\int_{\RR^3} U_{x_i,t}\ti U^\veps_{x_it}+c^2\nabla U_{x_i}
	\cdot\nabla \ti U^\veps_{x_i}\, dx\nn\\
	&=\mathcal E_{\del_i U}(0)+\mathcal E^\veps_{\del_i U^\veps}(0)
	- \int_{|x|>\veps}U_{x_i,t} U^\veps_{x_it}+c^2\nabla U_{x_i}
	\cdot\nabla U^\veps_{x_i}\, dx.
\end{align}
Differentiating in time and integrating by parts in the last integral, give
\beq\label{indiv_term}
	\dot{\mathcal E}_{\del_i U-\del_i \ti U^\veps}(t) 
	=c^2\int_{|x|=\veps}\big(U_{x_it}\big)\big(\del_r U^\veps_{x_i}\big)\, dS.
\eeq
Observing that we have
\[\sum_{i=1}^3U_{x_it}\del_r U^\veps_{x_i}=\big(\del_r U_t\big)\big(\del_{rr}U^\veps\big)\]
along $\{|x|=\veps\}$ (recall that $\del_r U^\veps$ vanishes along $\{|x|=\veps\}$),
we obtain from \eq{indiv_term} that 
\beq\label{2_nd_energy_diff}
	\EE^\veps(T)=\EE^\veps(0)+c^2\text{area}(B_\veps)
	\int_0^T \big(\del_r U_t(t,\veps)\big)\big(\del_{rr}U^\veps(t,\veps)\big)\, dt.
\eeq
To estimate $\mathcal E^\veps(T)$ and $\EE^\veps(T)$, and hence also 
$\mathcal D^\veps(T)$ according to \eq{1_vs_0_energy}, we employ 
the solution formulae \eq{U} and \eq{U_eps}.

\subsection{Initial differences in energy }
The details of estimating the initial differences of the first and second 
order energies, i.e.\ $\mathcal E^\veps(0)$ and $\EE^\veps(0)$, were 
carried out in Section 3.2 of \cite{jt1} (and makes use of 
\eq{init_approx_data}). Translating to our present notation 
we have that 
\beq\label{initial_0th_energy_diff}
	\mathcal D^\veps(0)\lea \veps^2\|\Phi\|_{H^1(B_{2\veps})}^2,
\eeq
\beq\label{initial_1st_energy_diff}
	\mathcal E^\veps(0)\lea \veps^2\|\Psi\|_{H^1(B_{2\veps})}^2
	+\|\Phi\|_{H^1(B_{2\veps})}^2,
\eeq
and 
\beq\label{initial_2nd_energy_diff}
	\EE^\veps(0)\lea \|\Psi\|_{H^1(B_{2\veps})}^2
	+\|\Phi\|_{H^2(B_{2\veps})}^2.
\eeq 

\subsection{Estimating growth of first order energy difference}
According to \eq{1_st_energy_diff}, to estimate $\mathcal E^\veps(T)$
we need to estimate the quantities $U^\veps_t$ and $\del_r U$ along $|x|=\veps$. 
For the remaining term involving $U_t(t,x)$ in \eq{1_st_energy_diff} (for $|x|\leq \veps$), 
it will suffice to employ an energy estimate that does not require formulae.

Before considering these terms in detail we record the following fact. 
For any $k\in \RR$ and for any $t>0$ let 
\[Q_k^\veps(t):=\frac{1}{\veps^k}\left(e^{-\frac{ct+\veps}{\veps}}
\int_\veps^{ct+\veps}\left[\frac{s\psi^\veps(s)}{c}
-\frac{s\vp^\veps(s)}{\veps}\right]e^{\frac{s}{\veps}}\, ds 
+ (ct+\veps)\vp^\veps(ct+\veps)\right);\]
then
\beq\label{Q}
	Q_k^\veps(t)\to 0\qquad\text{as $\veps\downarrow 0$}.
\eeq
To see this, integrate by parts in the $\vp^\veps$-term to get that
\begin{align*}
        Q_k^\veps(t)&=\frac{1}{c}
        \int_\veps^{ct+\veps}\!\!\!\! s\psi^\veps(s)\veps^{-k}e^{\frac{s-ct-\veps}{\veps}}\, ds
        +\vp^\veps(\veps)\veps^{1-k}e^{-\frac{ct}{\veps}}
        +\int_\veps^{ct+\veps}\!\!\!\!\left[\vp^\veps(s)
        +s{\vp^\veps}'(s)\right]\veps^{-k}e^{\frac{s-ct-\veps}{\veps}}\, ds.
\end{align*}
Recalling \eq{approx_neum_data_1}-\eq{approx_neum_data_2} and using that 
$\vp$ and $\psi$ are fixed, smooth functions, the Dominated Convergence Theorem 
yields $Q_k^\veps(t)\to 0$ as $\veps\downarrow 0$.

\subsubsection{Estimating $U^\veps_t(t,\veps)$}
According to \eq{U_eps_bndry_val} we have
\[U^\veps_t(t,\veps)=\frac{c}{\veps}\left(\vp^\veps(ct+\veps)
+(ct+\veps){\vp^\veps}'(ct+\veps)
+\frac{(ct+\veps)}{c}\psi^\veps(ct+\veps)\right)-cQ^\veps_2(t).\]
As $Q^\veps_2(t)$ tends to zero while $\vp^\veps$ and $\psi^\veps$
remain bounded, we conclude that
\beq\label{U^veps_t_bound}
	|U^\veps_t(t,\veps)|\lea \frac{1}{\veps}\qquad\text{for all $t\in[0,T]$ as $\veps\downarrow 0$.}
\eeq

\subsubsection{Estimating $\del_rU(t,\veps)$}
According to \eq{U} we have
\beq\label{U_r}
	\del_rU(t,\veps)=\left\{
	\begin{array}{ll}
		t\geq \frac{\veps}{c}: & -\frac{1}{2\veps^2}\left[ (ct+\veps) \vp(ct+\veps)-(ct-\veps)\vp(ct-\veps)\right]\\\\
					&+\frac{1}{2\veps}\left[\vp(ct+\veps)+ (ct+\veps)\vp'(ct+\veps)+\vp(ct-\veps)
					+(ct-\veps)\vp'(ct-\veps)\right]\\\\
					&-\frac{1}{2c\veps^2}\int_{ct-\veps}^{ct+\veps} s\psi(s)\, ds
					+\frac{1}{2c\veps}\left[(ct+\veps) \psi(ct+\veps)+(ct-\veps)\psi(ct-\veps)\right]\\\\\\
		t\leq \frac{\veps}{c}: & -\frac{1}{2\veps^2}\left[ (\veps+ct) \vp(\veps+ct)+(\veps-ct)\vp(\veps-ct)\right]\\\\
					&+\frac{1}{2\veps}\left[ \vp(\veps+ct) +(\veps+ct)\vp'(\veps+ct)+\vp(\veps-ct)
					+(\veps-ct)\vp'(\veps-ct)\right]\\\\
					& -\frac{1}{2c\veps^2}\int_{\veps-ct}^{\veps+ct} s\psi(s)\, ds
					+\frac{1}{2c\veps}\left[(\veps+ct) \psi(\veps+ct)-(\veps-ct) \psi(\veps-ct)\right].
	\end{array}
	\right.
\eeq
The terms for $t\geq\frac{\veps}{c}$  are estimated by 2nd order 
Taylor expansion of $\vp(ct\pm\veps)$ and $\psi(ct\pm\veps)$
about $\veps=0$. The terms for $t\leq \frac{\veps}{c}$ are estimated 
by 2nd order Taylor expansion of $\vp$ and $\psi$ about zero, 
and then using that $\vp'(0)=\psi'(0)=0$. (As observed earlier, this 
holds since $\vp$ and $\psi$ are profile functions of the {\em smooth}, 
radial functions $\Phi$ and $\Psi$, respectively). 
These expansions are straightforward and we omit them. The end 
result is that the leading order terms in \eq{U_r} cancel, leaving terms 
of size at most $O(\veps)$. We thus have that
\beq\label{U_r_bound}
	|\del_r U(t,\veps)|\lea \veps
	\qquad\text{for all $t\in[0,T]$ as $\veps\downarrow 0$.}
\eeq
(Note: this is actually obvious since we know that $U$ is a smooth, 
radial solution satisfying $\del_r U(t,0)\equiv 0$ and with fixed data 
independent of $\veps$.)

Finally, the Cauchy-Schwarz inequality and Lemma \ref{energy_1} 
give
\[\Big|\int_{|x|<\veps} U_t(t,x)\, dx\Big|
\lea \veps^\frac{3}{2}\mathcal E_U(0)^\frac{1}{2}. \]
Applying this together with \eq{U^veps_t_bound} and \eq{U_r_bound} 
in \eq{1_st_energy_diff}, we conclude that
\beq\label{1_st_energy_diff_1}
	\mathcal E^\veps(T) \lea \mathcal E^\veps(0) +\veps^\frac{1}{2}.
\eeq
%

\subsection{Estimating growth of second order energy differences}
Next, according to \eq{2_nd_energy_diff}, to estimate $\EE^\veps(T)$, 
we need to estimate the quantities $\del_rU_t$ and $\del_{rr} U^\veps$ 
along $|x|=\veps$.

\subsubsection{Estimating $\del_rU_t(t,\veps)$}
By taking the time derivative of \eq{U_r} and then Taylor expanding 
the various terms as outlined above, we deduce that 
\beq\label{U^_rt_bound}
	|\del_r U_t(t,\veps)|\lea \veps\qquad\text{for all $t\in[0,T]$ as $\veps\downarrow 0$.}
\eeq
\subsubsection{Estimating $\del_{rr}U^\veps(t,\veps)$}
This estimate again requires a direct, but rather long, calculation 
(which we omit), followed by a careful analysis of the resulting expression.

The first step is to calculate $\del_{rr}U^\veps(t,r)$ for $\veps\leq r\leq ct+\veps$, 
by using the first part of formula \eq{U_eps}. A number of cancellations occur 
when the resulting expression is evaluated at $r=\veps$, and we are left with
\begin{align*}
	\del_{rr}U^\veps(t,\veps) = Q^\veps_3(t)
	&-\frac{1}{\veps^2}\left[\vp^\veps+(ct-\veps){\vp^\veps}'
	-\veps(ct+\veps){\vp^\veps}''\right]
	-\frac{1}{c\veps^2}\left[ct\psi^\veps-\veps(ct+\veps){\psi^\veps}'\right],
\end{align*}
where $\vp^\veps$, $\psi^\veps$, and their derivatives are evaluated at $ct+\veps$.
According to \eq{approx_neum_data_1}-\eq{approx_neum_data_2} we have that 
$\vp^\veps$, $\psi^\veps$, and their first derivatives remain bounded independently
of $\veps$, while ${\vp^\veps}''$ is at most of order $\frac{1}{\veps}$.
Since $Q^\veps_3(t)\to 0$ as $\veps\downarrow 0$ by \eq{Q}, we therefore have that 
\beq\label{U^eps_rr_bound}
	|\del_{rr}U^\veps(t,\veps)|\lea \frac{1}{\veps^2}
	\qquad\text{for all $t\in[0,T]$ as $\veps\downarrow 0$.}
\eeq
Finally, by using \eq{U^_rt_bound} and \eq{U^eps_rr_bound} in 
\eq{2_nd_energy_diff}, we conclude that
\beq\label{2_nd_energy_diff_1}
	\EE^\veps(T) \lea \EE^\veps(0) +\veps.
\eeq

\subsection{Convergence of exterior Neumann solutions}
According to the definitions of $\mathcal D^\veps(t)$, $\mathcal E^\veps(t)$, and
$\EE^\veps(t)$, together with the estimates \eq{1_vs_0_energy}, 
\eq{1_st_energy_diff_1}, \eq{2_nd_energy_diff_1} we have
\begin{align*}
	\|U(t)-\ti U^\veps(t)\|_{H^2(\RR^3)}^2
	&\lea \mathcal D^\veps(t) + \mathcal E^\veps(t)+\EE^\veps(t)\\
	&\lea \mathcal D^\veps(0)+\mathcal E^\veps(0)+\EE^\veps(0)
	 +\veps^\frac{1}{2},
\end{align*}
at any time $t\in[0,T]$.
Applying the bounds \eq{initial_0th_energy_diff}, and \eq{initial_1st_energy_diff},
\eq{initial_2nd_energy_diff}, we conclude that the (extended) Neumann solutions
$\ti U^\veps(t)$ converge to the Cauchy solution $U(t)$ in $H^2(\RR^3)$, 
uniformly on bounded time intervals, as $\veps\downarrow 0$. Thanks to Proposition 
\ref{smooth_case}, this concludes
the proof of part (i) of Theorem \ref{main_result}.

\section{Exterior Dirichlet solutions}\label{gen_dir_soln}
In this and the next section $U^\veps$ refers to the exterior Dirichlet 
solutions; similarly for their extensions $\ti U^\veps(t,x)$.

For fixed initial data $\Phi,\, \Psi\in C^\infty_{c,\, rad}(\RR^3)$ and any 
$\veps>0$ we shall derive a formula for the exterior, radial Dirichlet 
solution $U^\veps(t,x)$, defined for $|x|\geq\veps$ and satisfying 
$U^\veps|_{r=\veps}=0$.  We refer to $U^\veps$ as the {\em exterior 
Dirichlet solution} corresponding to the solution $U$ of (CP) with data 
$\Phi$, $\Psi$. In Section \ref{comp_dir} we will then estimate how it 
(really, its extension $\tilde U^\veps(t,x)$ to all of $\RR^3$) approximates 
the solution $U(t)$ in $H^1(\RR^3)$ at fixed times.

To generate the exterior Dirichlet solution $U^\veps(t,x)$ and its 
extension we use the following scheme. 
To smoothly approximate the original data $(\Phi,\Psi)$ with exterior 
Dirichlet data we fix a smooth, nondecreasing cutoff function 
$\chi:\RR_0^+\to\RR_0^+$ with
\beq\label{eta_props}
	\chi\equiv 0 \quad\text{on $[0,1]$,}
	\quad \chi\equiv 1\quad\text{on $[2,\infty)$.}
\eeq
Then, with $\vp$ and $\psi$ as in \eq{rad_versns} we define
\beq\label{approx_dir_data_1}
	\Phi^\veps(x)\equiv \vp^\veps(|x|)
	:=\chi\big(\textstyle\frac{|x|}{\veps}\big)\vp(|x|)
\eeq
and
\beq\label{approx_dir_data_2}
	\Psi^\veps(x)\equiv \psi^\veps(|x|)
	:=\chi\big(\textstyle\frac{|x|}{\veps}\big)\psi(|x|).
\eeq
We refer to $(\Phi^\veps,\Psi^\veps)$ as the {\em Dirichlet data} 
corresponding to the original Cauchy data $(\Phi,\Psi)$ for (CP).
Note that the Dirichlet data are actually defined on all of $\RR^3$, 
that they vanish identically on $B_\veps$, and hence their restrictions 
to the exterior domain $\{x\in\RR^3\,:\, |x|\geq\veps\}$  
satisfy homogeneous Dirichlet conditions along $|x|=\veps$.

The exterior Dirichlet solution $U^\veps(t,x)$ is then the unique radial 
solution of the initial-boundary value problem
\[\left\{\begin{array}{ll}
	\square_{1+3}V =0 & \text{on $(0,T)\times\{|x|>\veps\}$}\\
	V(0,x)=\Phi^\veps(x) & \text{for $|x|>\veps$}\\
	V_t(0,x)=\Psi^\veps(x) & \text{for $|x|>\veps$}\\
	V(t,x)=0 & \text{along $|x|=\veps$ for $t>0$.}
\end{array} \right.\]

We next record the solution formula for the exterior, radial Dirichlet solution 
$U^\veps(t,r)$ (which is simpler to derive than the formula for the 
exterior Neumann solution):
\beq\label{U_eps_dir}
	U^\veps(t,r)=\left\{
	\begin{array}{ll}
		\veps\leq r\leq ct+\veps: & \frac{1}{2r}\left[ (ct+r) \vp^\veps(ct+r)
					-(ct-r+2\veps)\vp^\veps(ct-r+2\veps)\right]\\
				& +\frac{1}{2cr}\int_{ct-r+2\veps}^{ct+r} s\psi^\veps(s)\, ds\\\\
		r\geq ct+\veps: & \frac{1}{2r}\left[ (r+ct) \vp^\veps(r+ct)+(r-ct)\vp^\veps(r-ct)\right]\\
					& +\frac{1}{2cr}\int_{r-ct}^{r+ct} s\psi^\veps(s)\, ds.
	\end{array}
	\right.
\eeq
We finally extend $U^\veps$ at each time to obtain an everywhere defined approximation of the 
Cauchy solution $U$. The natural choice is to extend $U^\veps$ continuously as zero
on $B_\veps$ at each time:
\beq\label{tilde_U_eps_dir}
	\tilde U^\veps(t,x)=\left\{
	\begin{array}{ll}
		0 & \text{for $0\leq |x|\leq \veps$}\\\\
		U^\veps(t,x) & \text{for $|x|\geq \veps$.}
	\end{array}
	\right.
\eeq
We note that 
\beq\label{initial_approx_N_data}
	\ti U^\veps(0,x)=\Phi^\veps(x),\qquad \ti U_t^\veps(0,x)=\Psi^\veps(x)
	\qquad\text{for all $x\in\RR^3$.}
\eeq

\section{Comparing the Cauchy and exterior Dirichlet solutions}\label{comp_dir}
We proceed to estimating the $H^1$-distance
\[\|U(t)-\tilde U^\veps(t)\|_{H^1(\RR^3)},\]
and show that it vanishes as $\veps\downarrow 0$. 
As for exterior Neumann solutions we prefer to estimate this  
difference by estimating the first order energy 
\[\mathcal E^\veps(t)=\mathcal E_{U-\ti U^\veps}(t)\]
as defined in \eq{ult_energ_1}. This energy bounds the $L^2$-norm
of the gradient of the difference $U-\tilde U^\veps$, and it also 
controls the $L^2$-norm of $U-\tilde U^\veps$ itself. 
The calculations for these estimates are similar to the 
ones for the Neumann case in Section \ref{energies}, and will only be outlined. 

First, a direct calculation similar to what was done above (using that the
energies $\mathcal E_U(t)$ and $\mathcal E^\veps_{U^\veps}(t)$ are both 
conserved in time), shows that
\[\mathcal E^\veps(t)=\mathcal E^\veps_{U^\veps}(0)+\mathcal E_U(0)
-\int_{|x|>\veps}U_tU^\veps_t+c^2\nabla U^\veps\cdot\nabla  U\, dx.\]
Differentiating with respect to time, integrating by parts, and applying the 
Dirichlet condition for $U^\veps$ yield
\beq\label{E_dot}
	\dot{\mathcal E}^\veps(t)=c^2\int_{|x|=\veps}
	U_t(t,\veps)\del_r U^\veps(t,\veps)\, dS.
\eeq
Also, with $\mathcal D^\veps(t)$ defined as in \eq{L2_dist}, we have that 
\eq{1_vs_0_energy} holds also in the case of Dirichlet boundary conditions.

To estimate the $H^1$-distance between the Cauchy 
solution $U(t)$ and the exterior Dirichlet solution $\tilde U^\veps(t)$,
we proceed to provide bounds for the initial terms $\mathcal D^\veps(0)$ and 
$\mathcal E^\veps(0)$, as well as for the surface integral in \eq{E_dot}.
It is immediate to verify that  
\[\mathcal D^\veps(0)\lea \|\Phi^\veps
-\Phi\|_{L^2(\RR^3)}^2\leq \|\Phi\|_{L^2(B_{2\veps})}^2,\] 
and similarly that 
\[\mathcal E^\veps(0) \lea \|\Psi\|_{L^2(B_{2\veps})}^2
+\|\nabla \Phi^\veps-\nabla\Phi\|_{L^2(\RR^3)}^2.\]
To bound the last term we recall the definition of $\Phi^\veps$ 
in \eq{approx_dir_data_1} to calculate that
\begin{align}
	\|\nabla \Phi^\veps-\nabla\Phi\|_{L^2(\RR^3)}^2 
	&\lea \int_{|x|<2\veps}|\nabla \Phi|^2\, dx 
	+ \frac{1}{\veps^2}\int_{\veps<|x|<2\veps}|\Phi(x)|^2\, dx \nn\\
	&\lea \|\nabla \Phi\|_{L^2(B_{2\veps})}^2 
	+ \int_{\veps<|x|<2\veps}\frac{|\Phi(x)|^2}{|x|^2}\, dx.\label{intrmed_energy}
\end{align}
As $\Phi$ belongs to $H^1(\RR^3)$, Hardy's inequality (as formulated in 
Lemma 17.1 in \cite{tar}) shows that $\frac{|\Phi(x)|^2}{|x|^2}$ belongs to
$L^1(\RR^3)$, so that the Dominated Convergence Theorem yields
\[\|\nabla \Phi^\veps-\nabla\Phi\|_{L^2(\RR^3)}^2 \to0 \qquad\text{as $\veps\downarrow 0$.}\]
%
%
We have thus established that 
\beq\label{init_enegies}
	\mathcal D^\veps(0)\to 0\qquad\text{and}\qquad \mathcal E^\veps(0) 
	\to 0 \qquad\text{as $\veps\downarrow 0$.}
\eeq

\subsubsection{Estimating $\del_tU(t,\veps)$}
According to \eq{U} we have
\beq\label{U_t}
	\del_tU(t,\veps)=\left\{
	\begin{array}{ll}
		t\geq \frac{\veps}{c}: &\frac{c}{2\veps}\left[\vp(ct+\veps)+ (ct+\veps)\vp'(ct+\veps)-\vp(ct-\veps)
					-(ct-\veps)\vp'(ct-\veps)\right]\\\\
					&\frac{1}{2\veps}\left[(ct+\veps)\psi(ct+\veps)-(ct-\veps)\psi(ct-\veps)\right]\\\\\\
		t\leq \frac{\veps}{c}: &\frac{c}{2\veps}\left[ \vp(\veps+ct) +(\veps+ct)\vp'(\veps+ct)-\vp(\veps-ct)
					-(\veps-ct)\vp'(\veps-ct)\right]\\\\
					&+\frac{1}{2\veps}\left[(\veps+ct) \psi(\veps+ct)+(\veps-ct) \psi(\veps-ct)\right].
	\end{array}
	\right.
\eeq
We estimate the terms for $t\geq\frac{\veps}{c}$  by 2nd order 
Taylor expansion of $\vp(ct\pm\veps)$ and $\psi(ct\pm\veps)$
about $\veps=0$. The terms for $t\leq \frac{\veps}{c}$ are estimated 
by 2nd order Taylor expansion of $\vp$ and $\psi$ about zero, 
and then using that $\vp'(0)=\psi'(0)=0$. 
These expansions are straightforward and are omitted. The 
result is that the leading order term in \eq{U_t} for all times is $O(1)$. 
We thus have that
\beq\label{U_t_bound}
	|U_t(t,\veps)|\lea 1
	\qquad\text{for all $t\in[0,T]$ as $\veps\downarrow 0$.}
\eeq

\subsubsection{Estimating $\del_{r}U^\veps(t,\veps)$}
We first calculate $\del_{r}U^\veps(t,r)$ for $\veps\leq r\leq ct+\veps$
by using the first part of formula \eq{U_eps_dir}. Evaluating at $r=\veps$
gives that 
\begin{align*}
	\del_{r}U^\veps(t,\veps) = 
	\frac{1}{\veps}\left[\vp^\veps(ct+\veps)+(ct+\veps){\vp^\veps}'(ct+\veps)\right]
	+\frac{(ct+\veps)}{c\veps}{\psi^\veps}(ct+\veps).
\end{align*}
Recalling the definitions of $\vp^\veps$ and $\psi^\veps$ in 
\eq{approx_dir_data_1}-\eq{approx_dir_data_2}, and 
splitting the calculations into $t\gtrless \frac{\veps}{c}$, we obtain that
\beq\label{U^eps_r_bound}
	|\del_{r}U^\veps(t,\veps)|\lea \frac{1}{\veps}
	\qquad\text{for all $t\in[0,T]$ as $\veps\downarrow 0$.}
\eeq

\subsection{Convergence of exterior Dirichlet solutions}
By using \eq{U_t_bound} and \eq{U^eps_r_bound} in \eq{E_dot} we obtain that
\[|\dot{\mathcal E}^\veps(t)|\lea \veps \qquad\text{for all $t\in[0,T]$ as 
$\veps\downarrow 0$,}\]
such that \eq{init_enegies}${}_2$ gives 
\[\mathcal E^\veps(t)\to 0 \qquad\text{uniformly for $t\in[0,T]$ as 
$\veps\downarrow 0$.}\]
Finally, recalling that \eq{1_vs_0_energy} also holds in the Dirichlet case,
we have that \eq{init_enegies}${}_1$ yields
\[\mathcal D^\veps(t)\to 0 \qquad\text{uniformly for $t\in[0,T]$ as 
$\veps\downarrow 0$,}\]
as well. We thus conclude that 
\[\|U(t)-\tilde U^\veps(t)\|_{H^1(\RR^3)}\lea \mathcal D^\veps(t)+
\mathcal E^\veps(t)\to 0 \qquad\text{uniformly for $t\in[0,T]$ as 
$\veps\downarrow 0$.}\]
Thanks to Proposition \ref{smooth_case}, this concludes
the proof of part (ii) of Theorem \ref{main_result}.



\begin{bibdiv}
\begin{biblist}
\bib{bjs}{book}{
   author={Bers, Lipman},
   author={John, Fritz},
   author={Schechter, Martin},
   title={Partial differential equations},
   note={With supplements by Lars G\.arding and A. N. Milgram;
   With a preface by A. S. Householder;
   Reprint of the 1964 original;
   Lectures in Applied Mathematics, 3{\rm A}},
   publisher={American Mathematical Society, Providence, R.I.},
   date={1979},
   pages={xiii+343},
   isbn={0-8218-0049-3},
   review={\MR{598466}},
}
\bib{jt1}{article}{
   author={Jenssen, Helge Kristian},
   author={Tsikkou, Charis},
   title={Radial solutions to the Cauchy problem for $\square_{1+3}U=0$
   as limits of exterior solutions},
   journal={Submitted, available at https://arxiv.org/abs/1512.02297},
   date={2015},
}
\bib{ra}{book}{
   author={Rauch, Jeffrey},
   title={Partial differential equations},
   series={Graduate Texts in Mathematics},
   volume={128},
   publisher={Springer-Verlag, New York},
   date={1991},
   pages={x+263},
   isbn={0-387-97472-5},
   review={\MR{1223093 (94e:35002)}},
   doi={10.1007/978-1-4612-0953-9},
}
\bib{sel}{book}{
   author={Selberg, Sigmund},
   title={Lecture Notes, Math 632, PDE},
   publisher={Johns Hopkins University},
   date={2001},
}
\bib{tar}{book}{
   author={Tartar, Luc},
   title={An introduction to Sobolev spaces and interpolation spaces},
   series={Lecture Notes of the Unione Matematica Italiana},
   volume={3},
   publisher={Springer, Berlin; UMI, Bologna},
   date={2007},
   pages={xxvi+218},
   isbn={978-3-540-71482-8},
   isbn={3-540-71482-0},
   review={\MR{2328004 (2008g:46055)}},
}

\end{biblist}
\end{bibdiv}
\end{document}